\documentclass{amsart}
\usepackage{enumitem}

\usepackage{hyperref}
\usepackage{graphicx}
  \graphicspath{{./figs/}}

\usepackage{mathtools} 
\usepackage{amsthm} 

\usepackage{array}

\usepackage[labelfont=rm]{subcaption}

\newtheorem{theo}{Theorem}
\newtheorem*{theo*}{Theorem}
\newtheorem{lemm}[theo]{Lemma}
\theoremstyle{remark}
\newtheorem{rema}[theo]{Remark}

\newcommand{\ind}{\mathbf{1}}

\newcommand{\area}{\operatorname{Area}}
\renewcommand{\d}{\mathrm{d}}
\newcommand{\N}{\mathbb{N}}
\newcommand{\Z}{\mathbb{Z}}
\newcommand{\R}{\mathbb{R}}

\newcommand{\SL}{\mathrm{SL}}
\newcommand{\GL}{\mathrm{GL}}
\newcommand{\PSL}{\mathrm{PSL}}
\newcommand{\slr}{{\SL(2,\R)}}

\newcommand{\pslr}{{\PSL(2,\R)}}
\newcommand{\CP}{\mathbb{CP}^1}

\newcommand{\C}{\mathcal{C}}
\newcommand{\Q}{\mathcal{Q}}
\renewcommand{\H}{\mathcal{H}}
\newcommand{\hyp}{\mathcal{H}^{hyp}}
\newcommand{\M}{\mathcal{M}}

\renewcommand{\P}{\mathcal{P}}

\begin{document}
\title[A non-varying phenomenon]{A non-varying phenomenon with an application to the wind-tree model}

\author{Angel Pardo}
\address{%
  Institut Fourier\\%
  Universit\'e Grenoble Alpes\\%
  38058 Grenoble cedex 09\\%
  France%
}

\begin{abstract}
We exhibit a non-varying phenomenon for the counting problem of cylinders, weighted by their area, passing through two marked (regular) Weierstrass points of a translation surface in a hyperelliptic connected component $\mathcal{H}^{hyp}(2g-2)$ or $\mathcal{H}^{hyp}(g-1,g-1)$, $g > 1$. As an application, we obtain the non-varying phenomenon for the counting problem of (weighted) periodic trajectories on the Ehrenfest wind-tree model, a billiard in the plane endowed with $\mathbb{Z}^2$-periodically located identical rectangular obstacles.
\end{abstract}

\maketitle
\section{Introduction}

In this work, we study a refinement of the (area) Siegel--Veech constant, which governs the asymptotics on the counting of closed geodesics on flat surfaces.
One of the motivations is the study of periodic trajectories in the Eherenfest wind-tree model, a statistical physics model, given by playing billiards in the plane with a $\Z^2$-periodic array of identical rectangular obstacles, as in Figure~\ref{figu:WTM1} (see~\S\ref{sect:application} for more details).

\begin{figure}[ht]
\centering
\includegraphics[width=.75\textwidth,height=.2\textheight,keepaspectratio]{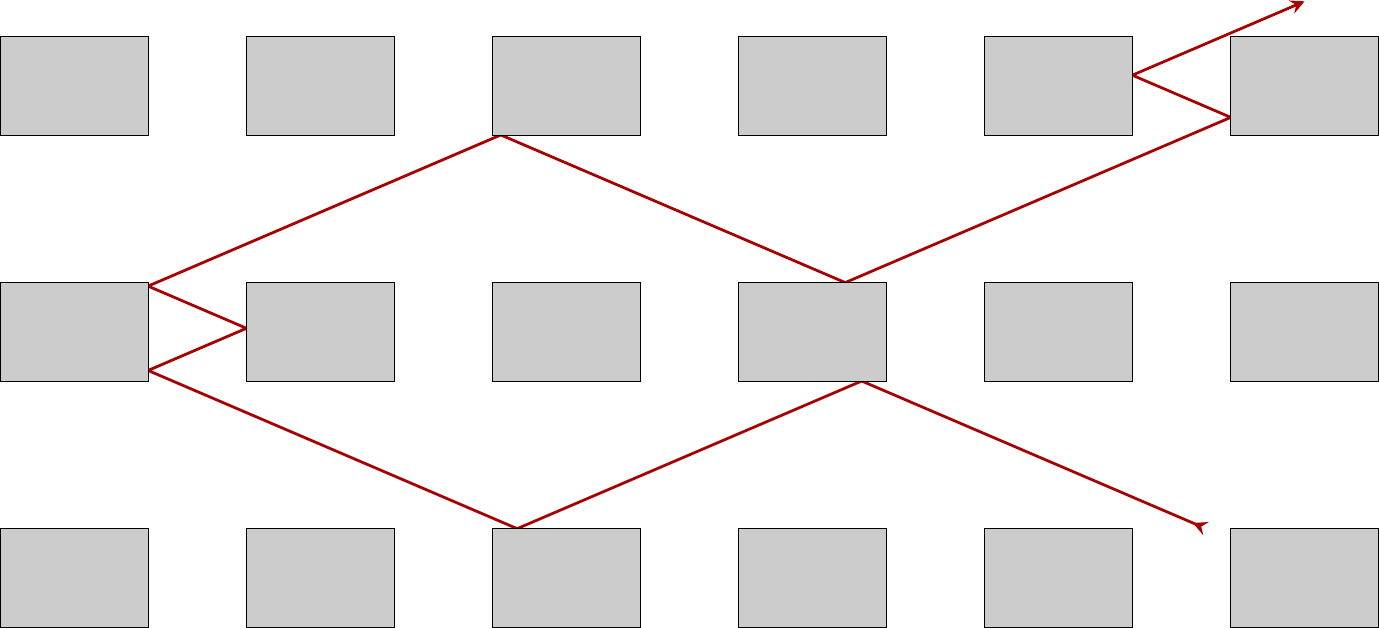}
\caption{The Ehrenfest wind-tree model.}
\label{figu:WTM1}
\end{figure}

The related counting problem has been widely studied in the context of (finite) rational billiards and compact flat surfaces, and it is related to many other questions such as the calculation of the volume of normalized strata (A.~Eskin, H.~Masur and A.~Zorich~\cite{EMZ}) or the sum of positive Lyapunov exponents of the geodesic Teichmüller flow (A.~Eskin, M.~Kontsevich and A.~Zorich~\cite{EKZ}) on strata of flat surfaces (Abelian or meromorphic quadratic differentials).

In this context, these flat surfaces correspond to translation 
surfaces or, equivalently, Abelian 
differentials on Riemann surfaces. Their moduli spaces are stratified according to the combinatorics of zeros of the corresponding Abelian 
differential and there is a natural action of $\SL(2,\R)$ 
on each stratum, which generalizes the action of $\SL(2,\R)$ on the space $\GL(2,\R)/\SL(2,\Z)$ of flat tori.

A connected component of a stratum of translation surfaces is said to be \emph{non-varying} if for every $\SL(2,\R)$-ergodic measure
in that component the above-mentioned sum of Lyapunov exponents is the same, that is, it does not depend on the measure. Such a non-varying phenomenon was observed numerically by M.~Kontsevich and A.~Zorich~\cite{Ko,KZ0} along with the initial observations on Lyapunov exponents for the Teichm\"uller geodesic flow.
Nowadays, there are two types of non-varying results. One for low genus, due to D.~Chen and M.~M\"oller~\cite{CM}, which uses a translation of the problem into algebraic geometry. The other one, for hyperelliptic loci, due to A.~Eskin, M.~Kontsevich and A.~Zorich~\cite{EKZ}, which is a consequence of their main result relating sum of Lyapunov exponents to Siegel--Veech constants.
In particular, the non-varying phenomenon for the sum of Lyapunov exponents is equivalent to the non-varying of Siegel--Veech constants.

Siegel--Veech constants are related to the counting of closed geodesics on translation surfaces. Closed geodesics form cylinders of parallel isotopic closed regular geodesics. Let $N(X,L)$ be the number of (maximal) cylinders of closed geodesics of length at most $L$, weighted by the area covered by the cylinder.

H.~Masur~\cite{Ma1,Ma2} proved that for every translation surface $X$, there are positive constants $c(X)$ and $C(X)$ such that \[c(X) L^2 \leq N(X,L) \leq C(X) L^2\]
for large enough $L$.
W.~Veech~\cite{Ve1} proved that there are in fact exact quadratic asymptotics for a particular type of translation surfaces, the Veech surfaces. A.~Eskin and H.~Masur~\cite{EMa} proved that for each $\SL(2,\R)$-ergodic probability measure $\mu$ on strata of normalized (area~$1$) translation surfaces, there is a constant $c(\mu)$ such that for almost every surface $X$, $N(X,L) \sim c(\mu)\cdot \pi L^2$.

The constant $c(\mu)$ is the Siegel--Veech constant; it is the constant in the Siegel--Veech formula, a Siegel-type formula introduced by W.~Veech~\cite{Ve3}, which can be translated into
\[c(\mu) = \frac{1}{\pi R^2} \int N(X,R)\d\mu(X).\]

The first explicit computations where made by W.~Veech~\cite{Ve1,Ve2}.
A.~Eskin, H.~Masur and A.~Zorich~\cite{EMZ} computed the Siegel--Veech constants for connected components of all strata of Abelian differentials, and also described all possible configurations of cylinders of closed geodesics which might be found on a generic translation surface.
In general, the particular constants for Veech surfaces do not coincide with the Siegel--Veech constants of the stratum where they live. Unless, of course, we face a non-varying phenomenon, as is for example the case of the hyperelliptic components $\hyp(2g-2)$ or $\hyp(g-1,g-1)$, $g>1$.

In this work we study a different but related counting problem: that of cylinders whose core curves pass through two marked regular Weierstrass points on hyperelliptic surfaces in a hyperelliptic component (see \S\ref{sect:hyperelliptic} and \S\ref{sect:Weierstrass}, for precise definitions).
In the generic situation, the Weierstrass points of a translation surface in an hyperelliptic component are symmetric and their numbering is irrelevant to specialized counting. However, as the following non-varying result shows, even in non-generic situations there is some ``hidden symmetry'' that makes the counting symmetric.

\begin{theo} \label{theo:main} 
Let $\mu$ be an $\slr$-ergodic probability measure supported in a hyperelliptic component $\hyp(2g-2)$ or $\hyp(g-1,g-1)$, $g>1$. Then, the (area) Siegel--Veech constant associated to the counting problem of cylinders whose core curve passes through two marked regular Weierstrass points does not depend on $\mu$ and is equal to
\[
\begin{dcases}
\frac{1}{\pi^2}\cdot\frac{1}{2g-1}, & \text{ if } \mathrlap{supp(\mu)\subset \hyp(2g-2),} \\
\frac{1}{\pi^2}\cdot\frac{1}{2g}, & \text{ if }  \mathrlap{supp(\mu)\subset \hyp(g-1,g-1).}
\end{dcases}
\]
\end{theo}

In the generic situation, that is, when $\mu$ is the Masur--Veech measure, this result is trivial since markings of regular Weierstrass points are symmetric and the (non-specialized) Siegel--Veech constants are known. Moreover, J.~Athreya, A.~Eskin and A.~Zorich~\cite{AEZ} independently computed this constants for every hyperelliptic loci (not only for hyperelliptic components), but again, only in the generic case.
Here, we prove the non-varying result and compute the exact value for \emph{every} $\slr$-ergodic probability measure on hyperelliptic components.
Moreover, our computations do not depend on the work of J.~Athreya, A.~Eskin and A.~Zorich~\cite{AEZ}.

It is a natural question whether this non-varying phenomenon takes place in every hyperelliptic locus as well, as is the case for the Siegel--Veech constant associated to the (non-specialized) counting of every cylinder.
We shall see that this is not true in general.
The main reason that makes our arguments work on hyperelliptic components only is that there, every cylinder passes through exactly two different regular Weierstrass points, and this is no longer true in any other hyperelliptic loci.

\subsection{Application to the wind-tree model}
\label{sect:application}

One of the motivations of this work is an application to the wind-tree model.
The wind-tree model corresponds to a billiard in the plane endowed with $\Z^2$-periodic obstacles of rectangular shape aligned along the lattice, as in Figure~\ref{figu:WTM1}. We denote by $\Pi(a,b)$ the wind-tree model whose obstacles have dimensions $(a,b)\in\left]0,1\right[^2$.

The wind-tree model (in a slightly different version) was introduced by P.~Ehrenfest and T.~Ehrenfest~\cite{EE} in 1912. J.~Hardy and J.~Weber \cite{HW} studied the periodic version. All these studies had physical motivations.

Several advances on the dynamical properties of the billiard flow in the wind-tree model were obtained using geometric and dynamical properties on moduli space of (compact) translation surfaces.
A.~Avila and P.~Hubert~\cite{AH} showed that for all parameters of the obstacle and for almost all directions, the trajectories are recurrent. There are examples of divergent trajectories constructed by V.~Delecroix~\cite{D}. The non-ergodicity was proved by K.~Fr\c{a}cek and C.~Ulcigrai~\cite{FU}. It was proved by V.~Delecroix, P.~Hubert and S.~Leli\`evre~\cite{DHL} that the diffusion rate is independent either on the concrete values of the parameters of the obstacle or on almost any direction and almost any starting point and is equals to $2/3$. A generalization of this last result was shown by V.~Delecroix and A.~Zorich \cite{DZ} for more complicated obstacles.

The result of V.~Delecroix, P.~Hubert and S.~Leli\`evre about the diffusion rate \emph{evince} a first non-varying phenomenon in the case of the classical wind-tree model, which corresponds to the `sum of Lyapunov exponents' counterpart. In this work we describe the `Siegel--Veech constant' counterpart of the non-varying phenomenon.

The author~\cite{Pa} studied the counting problem on wind-tree models proving that the number of periodic trajectories has quadratic asymptotic growth rate and computed, in the generic case, the Siegel--Veech constants for the classical wind-tree model as well as for the Delecroix--Zorich variant.
In this work we prove that, for the classical wind-tree model, this constant does not depend on the dimensions of the obstacles, exhibiting a non-varying phenomenon analogous to the one described above. More precisely, as a direct consequence of Theorem~\ref{theo:main}, we have the following.

\begin{theo} \label{theo:main wind-tree}
Let $N (\Pi(a,b),L)$ be the number of maximal families of isotopic periodic trajectories (up to $\Z^2$-translations) of length at most $L$ in $\Pi(a,b)$, weighted by the area covered by the family. Then,
\begin{itemize}[leftmargin=*]
\item For Lebesgue-almost every $a,b\in\left]0,1\right[$ and, in particular, if $a,b$ are rational or can be written as $1/(1-a) = x+z\sqrt{D}$ and $1/(1-b) = y + z\sqrt{D}$ with $x, y, z \in\mathbb{Q}$ and $x+y=1$ and $D$ a positive square-free integer, we have \[N (\Pi(a,b),L) \sim \frac{1}{3\pi^2} \cdot \frac{\pi L^2}{1-ab}.\]
\item In any other case, we have the weak asymptotic formula
\[ \int_0^L N (\Pi(a,b),\ell)\d\ell \sim 
  \frac{1}{3\pi^2} \cdot \frac{\pi L^2}{1-ab}.\]
\end{itemize}
\end{theo}

\begin{proof}
The statement is a compilation of several different results and is equivalent to $c (\Pi(a,b))=1/3\pi^2$ (cf.~\cite[Theorem~1.7]{AEZ} and \cite[Theorem~1.2]{Pa}).
By \cite[Corollary~5.6]{Pa}, the counting problem on the wind-tree model coincides with the counting problem of cylinders whose core curve passes through two marked regular Weierstrass points on a surface $L(a,b)\in\tilde\Q(1,-1^{5})=\H(2)$.

By elementary considerations on the Siegel--Veech formula (cf.~\cite[Lemma~1.1]{EKZ}) combined with the lifting properties of cylinders in $L(a,b)$ (see for example \cite[Lemma~3]{AH}), we have that $c (\Pi(a,b))$ coincides with the Siegel--Veech constant associated to the corresponding counting on $L(a,b)$.

Thus, by Theorem~\ref{theo:main}, we conclude that $c (\Pi(a,b))=1/3\pi^2$.
\end{proof}

\subsection{Strategy of the proof}
Hyperelliptic surfaces are orientation double covers of meromorphic quadratic differentials with at most simple poles on the Riemann sphere.
From a hyperelliptic surface $X$ in a hyperelliptic component $\hyp(2g-2)$ or $\hyp(g-1,g-1)$, $g>1$, and given two fixed regular Weierstrass points, we build three different translation surfaces which are covering of the original surface $X$. These coverings turn out to be hyperelliptic surfaces as well.
In order to treat both components at the same time, we work directly in the strata of quadratic differentials $\Q(d,-1^{d+4})$.
The problem of counting closed geodesics passing through two marked regular Weierstrass points reduces to the problem of counting saddle connections joining two marked poles.

The coverings mentioned above are double covers built by choosing two ramification points in $\CP$ and pulling back the quadratic differential in $\Q(d,-1^{d+4})$. Thus, by Riemann--Hurwitz formula, it is clear that they correspond to quadratic differentials on the Riemann sphere as well.

We introduce some configurations of cylinders associated to the marked poles and the ramification locus of these coverings, and describe the counting of cylinders having a saddle connection joining the two marked poles as a boundary components, in terms of one of these configurations.

We relate then the Siegel--Veech constants of the configurations of cylinders in $\Q(d,-1^{d+4})$ to their liftings on the coverings.
Decomposing the Siegel--Veech constants of the involved surfaces in terms of these configurations, we obtain a system of equations which allows us to describe the Siegel--Veech constants of the configurations in terms of those of the surfaces. 
Since the surfaces are genus zero, thanks to Eskin--Kontsevich--Zorich~\cite{EKZ}, the result is non-varying.
Describing the strata where the surfaces lie and putting the values of the corresponding Siegel--Veech constants in the expression allows us to compute explicitly the value of the Siegel--Veech constant associated to the configurations, and therefore, the one associated to the counting of cylinders bounded by the marked poles.
Coming back to the orientation double cover we obtain the result on the counting of cylinders whose core curve passes through the two Weierstrass points.

We present two families of counterexamples for hyperelliptic loci which are not hyperelliptic components. We exhibit hyperelliptic surfaces where the Siegel--Veech constant associated to the counting of cylinders whose core curve passes through two marked Weierstrass points does not coincide with the corresponding Siegel--Veech constant on the hyperelliptic loci where they lie. For this, in a first example, we use one of the covers defined above, which lies in a hyperelliptic locus which is not a hyperelliptic component. We find two families of pairs of regular Weierstrass points and show that the average Siegel--Veech constant associated to each family does not coincide.
Moreover, using a result of Athreya--Eskin--Zorich~\cite{AEZ}, we show that neither of this values coincide with the generic value on the stratum.

In a second example, we show that there are hyperelliptic surfaces with pairs of regular Weierstrass points which cannot be joined by closed geodesics. In particular, the relevant Siegel--Veech constant is zero in this case, while the generic value is strictly possitive.

For both families of examples we work again with quadratic differentials on $\CP$ and then deduce the corresponding result for hyperelliptic surfaces.

\subsection{Structure of the paper}
In \S\ref{sect:background} we briefly recall all the background necessary to formulate and prove the results. 
In \S\ref{sect:case} we prove Theorem~\ref{theo:main}. We describe the covering construction in \S\ref{sect:hyperelliptic coverings}. In \S\ref{sect:configurations} we introduce the associated configurations of cylinders, one of them coinciding with the one relevant to our problem.
We describe the system of equations that the associated Siegel--Veech constants satisfy and find the desired value.

We present in \S\ref{sect:counterexample} the two families of counterexamples, providing the values of (the average of) the pertinent Siegel--Veech constants for the counterexamples as well as for the generic case.

\subsection*{Acknowledgements}
The author is grateful to Pascal Hubert for useful discussions.
The author is thankful to the anonymous referee for comments and suggestions that improved the presentation of the paper.

\section{Background} \label{sect:background}

\subsection{Flat surfaces}
For an introduction and general references to this subject, we refer the reader to the surveys of Zorich~\cite{Zo}, Forni--Matheus~\cite{FM}, Wright~\cite{Wr2}.

\subsubsection{Flat surfaces and strata}
Let $g\geq 1$, $\{n_1,\dots,n_k\}$ be a partition of $2g-2$ and $\H(n_1,\dots,n_k)$ denote a stratum of Abelian differentials, that is, the space holomorphic $1$-forms on Riemann surfaces of genus $g$, with zeros of degree $n_1,\dots,n_k\in\N$.
There is a one to one correspondence between Abelian differentials and translation surfaces, surfaces which can be obtained by edge-to-edge gluing of polygons in $\R^2$ using translations only. Thus, we refer to elements of $\H(n_1,\dots,n_k)$ as translation surfaces.
A translation surface has a canonical flat metric, the one obtained form $\R^2$, with conical singularities of angle $2\pi(n+1)$ at zeros of degree~$n$ of the Abelian differential.

We also consider strata $\Q(d_1,\dots,d_k)$ of meromorphic quadratic differentials with at most simple poles on Riemann surfaces with zeros of order $d_1,\dots,d_k$, $d_i\in\{-1\}\cup\N$ for $i=1,\dots,k$ (in a slight abuse of vocabulary, we are considering poles as zeros of order~$-1$) and $\sum_{i=1}^k d_i=4g-4$.
A meromorphic quadratic differential also defines a canonical flat metric with conical singularities of angle $\pi(d+2)$ at zeros of order~$d$.

In this paper, a meromorphic quadratic differential is not the square of an Abelian differential. This condition is automatically satisfied if at least one of parameters $d_j$ is odd.

\paragraph{Notation}
As usual, we use ``exponential'' notation to denote multiple zeroes (or simple poles) of the same degree, for example $\Q(1,-1^5)=\Q(1,-1,-1,-1,-1,-1)$.

A flat surface is a Riemann surface with the flat metric corresponding to an Abelian or meromorphic quadratic differential.

\subsubsection{Canonical orientation double cover}
One can canonically associate with every meromorphic quadratic differential $q$ on a Riemann surface $S$ another connected curve with an Abelian differential on it. It is the unique double covering of $S$ (possibly ramified at singularities of $q$) such that the pullback of $q$ is the square of an Abelian differential.

\paragraph{Notation} We denote by $\tilde\Q(d_1,\dots,\d_k)$ the locus of translation surfaces consisting on the canonical orientating double cover of surfaces in the strata of half-translation surfaces $\Q(d_1,\dots,\d_k)$.

\subsubsection{Hyperelliptic surfaces, loci and components} \label{sect:hyperelliptic}
We say that a translation surface $X$ is a hyperelliptic surface if it corresponds to the canonical orientation double cover of a meromorphic quadratic differential on a Riemann surface of genus zero, that is, on the sphere $\CP$. Equivalently, if $X\in\tilde\Q(d_1,\dots,d_k)$ with $\sum_{j=1}^{k} d_j = -4$ and, in this case, we say that $\tilde\Q(d_1,\dots,d_k)$ is a hyperelliptic locus.

There is a series of hyperelliptic loci which plays a special role: $ \Q(d,-1^{d+4})$, for $d \geq 1$.
In these cases, the hyperelliptic loci coincides with a connected component of the corresponding stratum (see~\cite[\S2.1]{KZ}), the hyperelliptic compontent, which is denoted by
\[ \tilde\Q(d,-1^{d+4}) = \begin{cases} \hyp(2g-2), & \text{ if } d = 2g-3, \\ \hyp(g-1, g-1), & \text{ if } d = 2g-2. \end{cases}\]

\subsubsection{Weierstrass points} \label{sect:Weierstrass}
Every translation surface obtained as an orientation covering comes with an involution. In the case of hyperelliptic surfaces, we call it the hyperelliptic involution. The hyperelliptic involution of a hyperelliptic surface of genus $g$ has exactly $2g+2$ fixed points. These fixed points are called Weierstrass points. We say that a Weierstrass point is regular if it is regular for the flat metric, that is, if it is not a conical singularity. Note that regular Weierstrass points are exactly those points that projects to poles in the corresponding meromorphic quadratic differential on the sphere.

\subsection{Counting problem}
We are interested in the counting of closed geodesics of bounded length on translation surfaces.
Together with every closed regular geodesic in a translation surface $X$ we have a bunch of parallel closed regular geodesics. 
A cylinder on a flat surface is a maximal open annulus filled by isotopic simple closed regular geodesics. A cylinder $C$ is isometric to the product of an open interval and a circle, its core curve $\gamma_C$ is the geodesic projecting to the middle of the interval and its length $l(C)$ is the circumference of the circle.
A saddle connection is a geodesic joining two different singularities or a singularity to itself, with no singularities in its interior. Cylinders are always bounded by parallel saddle connections.

The number of cylinders of bounded length is finite. Thus, for any $L > 0$ the following quantity is well-defined:
\[N (X,L)=\frac{1}{\area(X)}\sum_{\substack{C\subset X \\ l(C) \leq L}} \area(C),\]
where the sum is over all cylinders $C$ in $X$ of length bounded by $L$.

\subsubsection{Siegel--Veech constant}
The following theorem is a special case of a fundamental result of Veech~\cite{Ve3}, considered by Vorobets in \cite{Vo2}.

\begin{theo*}[Veech] Let $\nu$ be an ergodic $\slr$-invariant probability measure on a stratum $\H_1(n_1,\dots,n_k)$ of Abelian differentials of area one. Then, the following ratio is constant (i.e. does not depend on the value of a positive parameter $R$):
\[ c (\nu) = \frac{1}{\pi R^2} \int N  (X, R) \d\nu.\]
\end{theo*}
This is called the Siegel--Veech formula, and the corresponding constant $c (\nu)$ is the Siegel--Veech constant.

A fundamental result of Eskin--Mirzhakani--Mohammadi~\cite{EMM} says that every $\slr$-orbit closure $\M$ is an affine invariant manifold and, in paricular, it is the support of an affine invariant measure $\nu_\M$ (see~\cite{EMM,EMi} for the precise definitions). For simplicity, we denote $c (\M)=c (\nu_\M)$.

We call a configuration of cylinders on an affine invariant manifold $\M$, a continuous $\slr$-equivariant application $\C$ which associates to $X\in\M$ (or any finite cover of $\M$) a collection of cylinders in $X$ (cf.~\cite{EMZ}). The previous discussion on the counting problem and Siegel--Veech constants applies as well in the case of configurations of cylinders and we denote by $c (\M,\C)$ the corresponding Siegel--Veech constant.

\paragraph{Notation} For a translation surface $X$, we denote by $c (X)$, the Siegel--Veech constant associated to the affine invariant measure $\nu_\M$ supported on its $\slr$-orbit closure $\M=\overline{\slr X}$. That is \[c (X)\coloneqq c (\M) = \frac{1}{\pi R^2}\int_{\M} N (Y,R)\d\nu_{\M}(Y).\]
Similarly, for a configuration of cylinders $\C$ defined on $\M$, we denote by $c (X,\C)$, the corresponding Siegel--Veech constant, $c (X,\C)=c (\M,\C)$.

As we are interested on Siegel--Veech constants associated to hyperelliptic surfaces, it is useful to relate them with the corresponding Siegel--Veech constants on the sphere they cover. We have the following general result, valid for any orientation double cover, due to Eskin--Kontsevich--Zorich~\cite[Lemma~1.1]{EKZ}.

\begin{lemm}\label{lemm:c=2c}
Let $X\in\tilde\Q(d_1,\dots,d_k)$ obtained from $Y\in\Q(d_1,\dots,d_k)$ by the orientation double cover construction. Then, $c (X) = 2c (Y)$.
\end{lemm}

\subsection{Non-varying phenomenon on hyperelliptic loci}\label{sect:non-varying}
The following result summarize the non-varying phenomenon for Siegel--Veech constants observed on hyperelliptic loci by Eskin--Kontsevich--Zorich~\cite[Theorem~3 and Lemma~1.1]{EKZ}.

\begin{theo}[Eskin--Kontsevich--Zorich] \label{theo:EKZ}
Let $X$ be a hyperelliptic surface such that the quotient sphere belongs to $\Q(d_1,\dots,d_k)$. That is, $X\in\tilde\Q(d_1,\dots,d_k)$ with $\sum_{j=1}^{k} d_j = -4$. Then
\begin{equation}\label{equa:EKZ}
c (X) = -\frac{1}{4\pi^2}\sum_{j=1}^{k} d_j \frac{d_j+4}{d_j + 2}.
\end{equation}
\end{theo}

It is clear, by Lemma~\ref{lemm:c=2c}, that the non-varying phenomenon on Theorem~\ref{theo:EKZ} is also valid for genus zero translation surfaces, with half the value of the Siegel--Veech constant.

\section{Refined non-varying phenomenon on hyperelliptic components}
\label{sect:case}

In this section we prove Theorem~\ref{theo:main} working directly on strata of meromorphic quadratic differentials on $\CP$. The conclusion on the hyperelliptic components can then be deduced using Lemma~\ref{lemm:c=2c}.

Thus, let $d\geq 1$ and fix $Y\in\Q(d,-1^{d+4})$, so that its orientation double cover $X$ belongs to the hyperelliptic component $\tilde\Q(d,-1^{d+4})$ (see \S\ref{sect:hyperelliptic}).

We want to study the Siegel--Veech constant associated to the counting of cylinders in $X$ which pass through two fixed regular Weierstrass points, say $w_1,w_2\in X$, as in Figure~\ref{figu:pocket-cylinder}~(left). These correspond to the preimages of two poles, say $p_1,p_2\in Y$ (see \S\ref{sect:Weierstrass}). Thus, the problem is reduced to study the Siegel--Veech constant associated to the counting of cylinders in $Y$ with a saddle connection joining $p_1$ and $p_2$ on one of its boundaries as in Figure~\ref{figu:pocket-cylinder}~(right).

\begin{rema}
\label{rema:every-cylinder}
It is worth to mention that since $Y$ has only one zero, for every cylinder on $Y$, one of its boundary components contains only poles and therefore, it is necessarily a saddle connection joining two different poles.
In particular, this means that in hyperelliptic components, every cylinder pass through exactly two regular Weierstrass points, which are contained in its core curve.
\end{rema}

\begin{figure}[ht]
\centering
\includegraphics[width=.9\textwidth,height=.25\textheight,keepaspectratio]{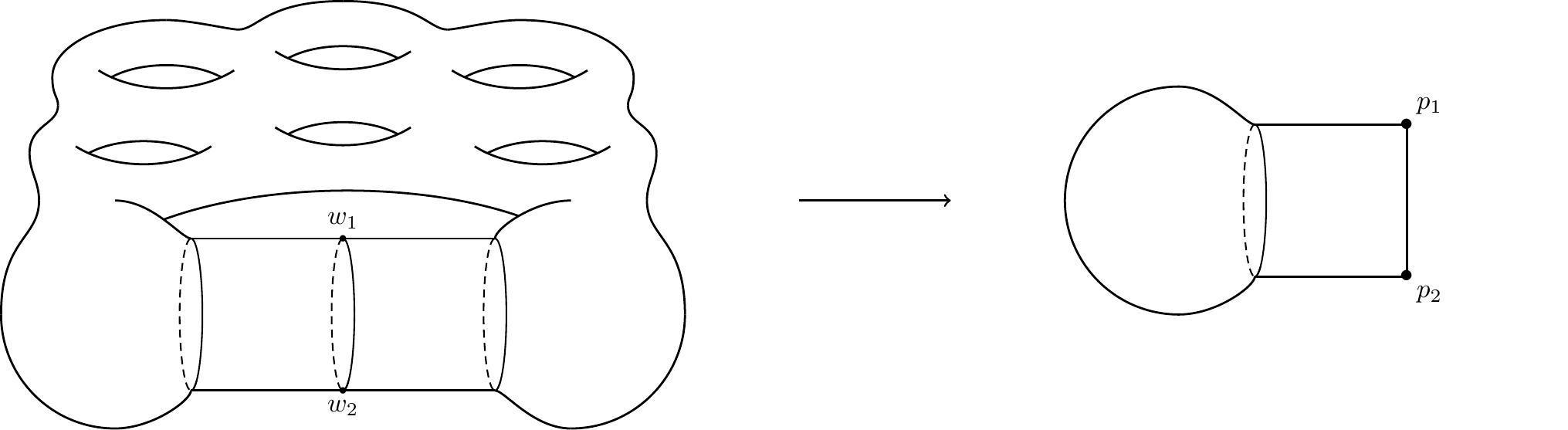}
\caption{A cylinder on $X$ whose core curve passes through two marked regular Weirestrass points (left) corresponds to cylinders on $Y$ for which one of its boundary components is formed by a saddle connection joining the two corresponding poles (right).}
\label{figu:pocket-cylinder}
\end{figure}

From the previous discussion, it follows that Theorem~\ref{theo:main} is equivalent to the following.
\begin{theo} \label{theo:main-sphere} 
Let $Y\in\Q(d,-1^{d+4})$. Then, the (area) Siegel--Veech constant on $\overline{\pslr Y}$ associated to the counting problem of cylinders with one boundary component being a saddle connection joining two marked poles, does not depend on $Y$ and is equal to
$\displaystyle\frac{1}{2\pi^2}\cdot\frac{1}{d+2}$.
\end{theo}

The rest of this section is devoted to prove Theorem~\ref{theo:main-sphere}, which in turns implies Theorem~\ref{theo:main}

\subsection{Coverings} \label{sect:hyperelliptic coverings}
Let $p_1,p_2\in Y$ be two poles and $z$ the degree~$d$ zero of the meromorphic quadratic differential defined by $Y$.
Define the following coverings of $Y$:
\begin{enumerate}
\item $Y_{01}$ is the double cover of $Y$ ramified at $z$ and $p_1$.
\item $Y_{10}$ is the double cover of $Y$ ramified at $z$ and $p_2$.
\item $Y_{11}$ is the double cover of $Y$ ramified at $p_1$ and $p_2$.
\end{enumerate}
By Riemann--Hurwitz formula, these are genus zero. Moreover, pulling back the meromorphic quadratic differential on $Y$ to these covers, it is easy to see that $Y_{01},Y_{10}\in \Q(2d+2,-1^{2(d+3)})$ and $Y_{11}\in\Q(d,d,-1^{2(d+2)})$.

\subsection{Configurations and Siegel--Veech constants} \label{sect:configurations}

We are concerned with the counting of cylinders bounded by a saddle connection joining the two fixed poles $p_1,p_2\in Y$.
Recall that since $Y$ has only one zero, for every cylinder on $Y$, one of its boundary components is necessarily a saddle connection joining two poles (see Remark~\ref{rema:every-cylinder}).
Thus, for a cylinder $C$ in $Y$, we denote by $\P(C)$ the set of these two poles.
Moreover, we define the profile of $C$ to be the couple $(\ind_{\P(C)}(p_1),\ind_{\P(C)}(p_2))\in\{0,1\}^2$. We also consider $\C_{pq}$ to be the configuration of cylinders in $Y$ of profile $(p,q)\in\{0,1\}^2$.

Then, all is reduced to the study of the Siegel--Veech constant $c (Y,\C_{11})$ associated to the configuration $\C_{11}$ on $Y$.
Denote by $c_{pq} = c (Y,\C_{pq})$. It is clear that \[c (Y) = c_{00} + c_{10} + c_{01} + c_{11}.\]

Now, consider the configuration $\C^{ij}_{pq}$ of cylinders $C$ in $Y_{ij}$ such that the projection to $Y$ is in $\C_{pq}$. Again, it is clear that $c (Y_{ij})$ decomposes into the sum of the Siegel--Veech constants $c (Y_{ij},\C^{ij}_{pq})$, $(p,q)\in\{0,1\}^2$.

The following general result relates the Siegel--Veech constants of configurations of cylinders on a double covering to the constant on the base space (see \cite[Lemma~4.1]{DZ}, which is a slight generalization of \cite[Lemma~1.1]{EKZ}; cf. Lemma~\ref{lemm:c=2c}).

\begin{lemm} \label{lemm:relation}
Let $\hat Z\to Z$ be any double covering of translation surfaces, $\C$ be a configuration of cylinders on $Z$ and $\hat\C$ be the lift of this configuration to $\hat Z$.
If the core curve of every cylinder in $\C$ has non-trivial monodromy, then $c (\hat Z,\hat \C) = c (Z,\C)/2$. 
If the core curve of every cylinder in $\C$ has trivial monodromy, then $c (\hat Z,\hat \C) = 2c (Z,\C)$.
\end{lemm}

Note that, in our case, the monodromy of a cylinder $C$ in $Y$ for the covering $Y_{ij}$ can be easily deduced in terms of the number $r_{ij}(C)$ of ramification points in $\P(C)$. In fact, by topological considerations, one can easily show that $C$ has trivial monodromy if and only if $r_{ij}(C) \in\{0,2\}$ (see e.g. \cite[Lemma~6.2]{Pa}), as in Figure~\ref{figu:cover-monodromy}.

\begin{figure}[ht!]\centering
\hfill
\begin{subfigure}[t]{.3\textwidth}
\centering
\smash{
\llap{
\smash{
\llap{\raisebox{-1\baselineskip}{$Y$}}
\llap{\raisebox{-6.5\baselineskip}{$Y_{ij}$}}
}}
\raisebox{-\height}{
\makebox[.9\textwidth]{
\includegraphics[scale=.75]{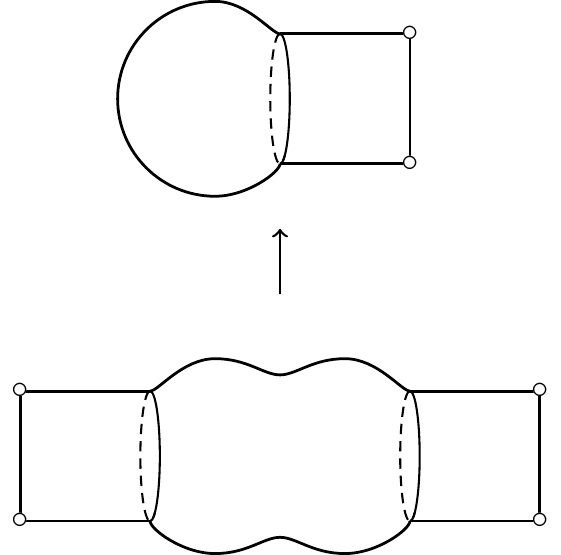}
}}}
\rule{0pt}{.23\textheight}
  \caption{$r_{ij}=0$.}
  \label{sfig:r0}
\end{subfigure}
\hfill
\begin{subfigure}[t]{.3\textwidth}
\centering
\smash{
\raisebox{-\height}{
\makebox[.9\textwidth]{
\includegraphics[scale=.75]{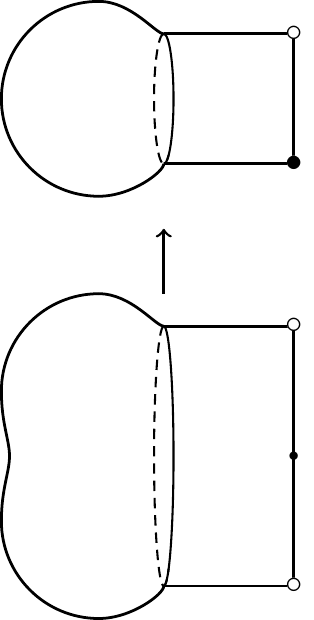}
}}}
\rule{0pt}{.23\textheight}
  \caption{$r_{ij}=1$.}
  \label{sfig:r1}
  \end{subfigure}
\hfill
\begin{subfigure}[t]{.3\textwidth}
\centering
\smash{
\raisebox{-\height}{
\makebox[.9\textwidth]{
\includegraphics[scale=.75]{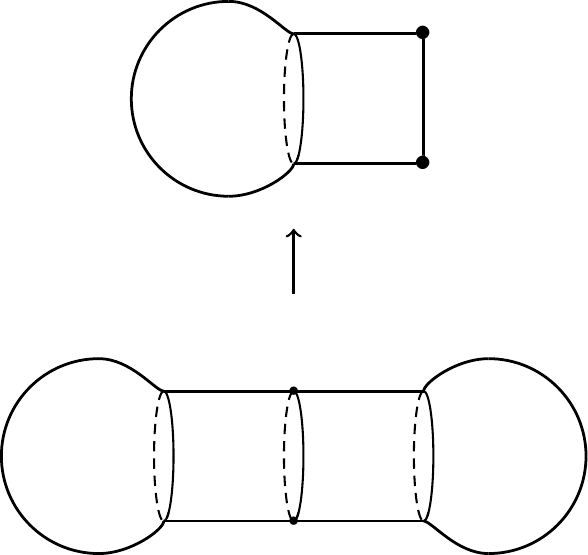}
}}}
\rule{0pt}{.23\textheight}
  \caption{$r_{ij}=2$.}
  \label{sfig:r2}
\end{subfigure}
\hfill
\caption{Possible liftings of a cylinder $C$ in $Y$ to $Y_{ij}$ depending on the number $r_{ij}$ of ramification points on its boundary. Ramification points are marked $\bullet$ and unramified, $\circ$.}
\label{figu:cover-monodromy}
\end{figure}

In other words, cylinders have monodromy $r_{ij}\mod2$.
It is clear that for cylinders in $\C^{ij}_{pq}$ we have $r_{ij}=ip+jq$, and thus, monodromy equals to $ip+jq\mod 2$. Using this, by the previous lemma, we have the following:
\[
\begin{array}{ r @{{}={}} r @{{}+{}} r @{{}+{}} r @{{}+{}} r }\label{eqsys}
c (Y) & c_{00} & c_{01} & c_{10} & c_{11}, \\[1em]
c (Y_{01}) & 2c_{00} & 2c_{01} & \dfrac{1}{2}c_{10} & \dfrac{1}{2}c_{11}, \\[1em]
c (Y_{10}) & 2c_{00} & \dfrac{1}{2}c_{01} & 2c_{10} & \dfrac{1}{2}c_{11}, \\[1em]
c (Y_{11}) & 2c_{00} & \dfrac{1}{2}c_{01} & \dfrac{1}{2}c_{10} & 2c_{11}.
\end{array}
\]

Moreover
\[
\begin{pmatrix}
1 & 1 & 1 & 1 \\
2 & 2 & 1/2 & 1/2 \\
2 & 1/2 & 2 & 1/2 \\
2 & 1/2 & 1/2 & 2
\end{pmatrix}^{-1}
= \; \frac{1}{3}
\begin{pmatrix*}[r]
-3 & 1 & 1 & 1 \\
2 & 1 & -1 & -1 \\
2 & -1 & 1 & -1 \\
2 & -1 & -1 & 1
\end{pmatrix*}
\]
and therefore, 
\begin{equation} \label{equa:c11}
c_{11} = \frac{1}{3}\left[2c (Y) - c (Y_{01}) - c (Y_{10}) + c (Y_{11})\right].
\end{equation}

This already gives the non-varying phenomenon for $c_{11}$. In fact, all the involved Siegel--Veech constants in the right-hand side of equation~\eqref{equa:c11} are non-varying, since all the involved surfaces are genus zero (see~\S\ref{sect:non-varying}).

To conclude the proof, it is enough to compute the Siegel--Veech constants for each surface.
This can be done by applying formula~\ref{equa:EKZ} in Theorem~\ref{theo:EKZ}. Recall that $Y\in\Q(d,-1^{d+4})$, $Y_{01},Y_{10}\in \Q(2d+2,-1^{2d+6})$ and $Y_{11}\in\Q(d,d,-1^{2d+4})$. Thus,

\[c (Y) = -\frac{1}{8\pi^2}\left(d\frac{d+4}{d+2} - 3\cdot(d+4)\right)=\frac{d+3}{4\pi^2}\cdot\frac{d+4}{d+2},\]
\[c (Y_{10}) = c (Y_{01}) = -\frac{1}{8\pi^2}\left((2d+2)\frac{2d+6}{2d+4} - 3\cdot(2d+6)\right)=\frac{d+3}{4\pi^2}\cdot\frac{2d+5}{d+2}\]
and
\[c (Y_{11}) = -\frac{1}{8\pi^2}\left(d\frac{d+4}{d+2}\cdot2 - 3\cdot(2d+4)\right)=\frac{1}{2\pi^2}\cdot\frac{d^2+4d+6}{d+2}.\]

Finally, plugging this in formula~\eqref{equa:c11}, we obtain
\[c_{11} = \frac{1}{3\pi^2}\cdot\frac{1}{d+2}\left[(d+3)(d+4)-(d+3)(2d+5)+(d^2+4d+6)\right]=\frac{1}{2\pi^2}\cdot\frac{1}{d+2}.\]
That is, the Siegel--Veech constant associated to the counting of cylinders on $Y$ for which one of its boundaries is a saddle connection joining the poles $p_1$ and $p_2$ equals $\displaystyle\frac{1}{2\pi^2}\cdot\frac{1}{d+2}$.
Proving Theorem~\ref{theo:main-sphere}.
\qed

Theorem~\ref{theo:main} follows then directly by lifting this to the orientation double cover $X$ of $Y$.
By Lemma~\ref{lemm:c=2c}, we get that the Siegel--Veech constant associated to the counting of cylinders on $X$ whose core curve passes through the two regular Weierstrass points $w_1$ and $w_2$ equals 
\[\frac{1}{\pi^2}\cdot\frac{1}{d+2} = 
\begin{dcases}
\frac{1}{\pi^2}\cdot\frac{1}{2g-1}, & \text{ if } X\in\hyp(2g-2), \\
\frac{1}{\pi^2}\cdot\frac{1}{2g}, & \text{ if }  X\in\hyp(g-1,g-1).
\end{dcases}
\]
\qed

\section{Counterexamples}
\label{sect:counterexample}

In this section we present two families of counterexamples: we exhibit hyperelliptic surfaces where the Siegel--Veech constant associated to the counting of cylinders whose core curve passes through two marked Weierstrass points does not coincide with the corresponding Siegel--Veech constant on the hyperelliptic loci where they lie.

For this, we follow the same strategy of the previous section, showing this on strata of meromorphic quadratic differentials on $\CP$.

\subsection{Different values in the same orbit closure}

Let $Y\in\Q(d,-1^{d+4})$ and consider the surface $Y_{11}$ from \S\ref{sect:hyperelliptic coverings}
, $Y_{11}\in\Q(d,d,-1^{2d+4})$.
In $Y_{11}$, cylinders joining poles fall into the configurations $\C_{00}^{11}$, $\C_{01}^{11}$, and $\C_{10}^{11}$ (see \S\ref{sect:configurations} for the precise definition) and every cylinder in these configurations joins two poles. Cylinders in $\C_{11}^{11}$ do not join poles (see Figure~\ref{sfig:r2}).

For a pole $p\in Y$, $p\neq p_1,p_2$, denote by $p',p''$ the corresponding preimages in $Y_{11}$, which are also poles.
Now, note that cylinders in $\C_{01}^{11}\cup\C_{10}^{11}$ are exactly those joining the pairs of poles in $Y_{11}$ of the form $p',p''$ (see Figure~\ref{sfig:r1}). Since there are $d+2$ such pairs, in average, the corresponding Siegel--Veech constant is
\[
\frac{1}{d+2}c \left(Y_{11},C_{01}^{11}\cup\C_{10}^{11}\right)
 = \frac{1}{d+2}\left( \frac{1}{2}c_{01} + \frac{1}{2}c_{10} \right) = \frac{1}{2\pi^2}\cdot\frac{1}{d+2},
\]
where the first equality holds by Lemma~\ref{lemm:relation} and the last one, by the equation system in the previous section, which gives $c_{01}=c_{10}=1/2\pi^2$.

On the other hand, we have that, in average, the Siegel--Veech constant associated to cylinders from $\C_{00}^{11}$ is 
\[
\frac{1}{2(d+2)(d+1)}c \left(Y_{11},\C_{00}^{11}\right) 
  = \frac{1}{2(d+2)(d+1)} 2c_{00} 
  = \frac{1}{4\pi^2}\cdot\frac{1}{d+2},
\]
where, again, the first equality holds by Lemma~\ref{lemm:relation} and the last one, by the equation system in the previous section, which gives $c_{00}=(d+1)/4\pi^2$. In addition, $2(d+2)(d+1) = \binom{2d+4}{2} - (d+2)$ is the number of possible pairs of poles joined by cylinders in $\C_{00}^{11}$.

Thus, we have shown that in $Y_{11}$ there are two families of pairs of poles such that, in average, the corresponding Siegel--Veech constants do not coincide, so that the non-varying phenomenon does not hold, not even inside the $\pslr$-orbit closure of $Y_{11}$.
\qed

Moreover, in the generic case, by a result of Athreya--Eskin--Zorich~\cite[Corollary~4.7 and Proposition~4.9]{AEZ}, we have that the Siegel--Veech constant associated to the counting of cylinders bounded by a saddle connection joining two marked poles on $\Q(d_1,\dots,d_k)$ equals $\displaystyle\frac{1}{2\pi^2}\cdot\frac{1}{k-3}$.
Thus, in $\Q(d,d,-1^{2d+4})$, the generic value is $\displaystyle\frac{1}{2\pi^2}\cdot\frac{1}{2d+3}$, which do not coincide with neither of the above averages.

It is worth to mention that, nevertheless, the whole average in $Y_{11}$, that is, the average of the Siegel--Veech constants for every possible pair of poles in $Y_{11}$ does coincide with the generic value. That is,
\[\frac{1}{\binom{2d+4}{2}} c \left(Y_{11},C_{00}^{11}\cup\C_{01}^{11}\cup\C_{10}^{11}\right)
  = \frac{1}{\binom{2d+4}{2}} \left(2c_{00} + \frac{1}{2} c_{01}+ \frac{1}{2} c_{10}\right)
  = \frac{1}{2\pi^2}\cdot\frac{1}{2d+3}\]

It is then a natural question whether the non-varying phenomenon holds at least in average in every hyperelliptic locus (recall that in hyperelliptic loci, different from hyperelliptic components, there are cylinders which do not pass through regular Weierstrass points).

\subsection{Some pairs of regular Weierstrass point are not even joinable by cylinders}

Let $Y\in\Q(d_1,\dots,d_k)$ with $\sum_{j=1}^{k} d_j = -4$ and such that at least two $d_j$ are not equal to $-1$, say $d_1,d_2\neq-1$. In particular, $\tilde\Q(d_1,\dots,d_k)$ is not a hyperelliptic component.
Let $z_1$ and $z_2$ be the zeros of degree $d_1$ and $d_2$ on $Y$, respectively, and consider $\hat Y$ the double cover of $Y$ ramified at $z_1$ and $z_2$. Again, by Riemann--Hurwitz formula, $\hat Y$ is genus zero and pulling back the meromorphic quadratic differential on $Y$ to $\hat Y$, we have that $\hat Y\in\Q(2d_1+2,2d_2+2,d_3,d_3,\dots,d_k,d_k)$.

Let $p\in Y$ be any pole. Since $\hat Y\to Y$ is not ramified over poles, $p$ has two preimages, say $p_1,p_2\in\hat Y$.

\begin{lemm}
There are no cylinders on $\hat Y$ with one boundary being a saddle connections joining $p_1$ and $p_2$.
\end{lemm}

\begin{proof}
Let $\hat s$ be a saddle connection in $\hat Y$ joining $p_1$ and $p_2$, and let $s$ be its image in $Y$.
Suppose that there is a cylinder in $\hat Y$ with $\hat s$ as one of its boundary component. Then, its projection to $Y$ is also a cylinder which has $s$ as one of its boundary components. But $s$ is a saddle connection joining $p$ to itself. In particular, $s$ cannot be a boundary component of a cylinder.
Thus, no such cylinder can exist.
\end{proof}

In particular, the associated Siegel--Veech constant is zero.
But, in the generic case (by the above-mentioned result of Athreya--Eskin--Zorich~\cite{AEZ}), the corresponding Siegel--Veech constant is strictly possitive.
\qed


\end{document}